\documentclass[12pt,a4paper]{article}
\usepackage[T1]{fontenc}
\usepackage[english]{babel}
\usepackage[centertags]{amsmath}
\usepackage{amsfonts}
\usepackage{amssymb}
\usepackage{amsthm}
\usepackage{newlfont}
\usepackage{graphicx}
\usepackage{color}
\usepackage{indentfirst}
\usepackage{float}
\usepackage[inline]{enumitem}
\usepackage{url}
\usepackage{mathtools}

\newtheorem{thm}{Theorem}
\newtheorem{lem}[thm]{Lemma}
\newtheorem{cor}[thm]{Corollary}
\newtheorem{prop}[thm]{Proposition}

\theoremstyle{definition}
\newtheorem{rem}[thm]{Remark}

\newcommand{\blank}{{\mspace{1mu}\cdot\mspace{1mu}}}
\newcommand{\ZZ}{\mathbb{Z}}
\newcommand{\KK}{\mathbb{K}}
\newcommand{\FF}{\mathbb{F}}
\newcommand{\QQ}{\mathbb{Q}}
\newcommand{\RR}{\mathbb{R}}
\newcommand{\acts}{{\mspace{1mu}\curvearrowright\mspace{1mu}}}

\DeclareMathOperator{\Iso}{\mathrm{Iso}}

\DeclareMathOperator{\pr}{\mathrm{pr}}
\DeclareMathOperator{\supp}{\mathrm{supp}}
\DeclareMathOperator{\LS}{\mathrm{LS}}
\DeclareMathOperator{\ch}{\mathrm{char}}
\DeclareMathOperator{\Stab}{\mathrm{Stab}}
\DeclareMathOperator{\GL}{\mathrm{GL}}
\DeclareMathOperator{\SL}{\mathrm{SL}}
\DeclareMathOperator{\SA}{\mathrm{SA}}
\DeclareMathOperator{\GA}{\mathrm{GA}}
\DeclareMathOperator{\Fix}{\mathrm{Fix}}

\title{The Banach--Tarski paradox in complete discretely valued fields}

\author{Kamil Orzechowski\\
University of Rzesz\'ow\\ 
35-959 Rzesz\'ow, Poland\\
\texttt{kamilo@dokt.ur.edu.pl}}

\date{July 15, 2026}

\begin{document}

\maketitle

\begin{abstract}
  We prove some results related to the classical Banach--Tarski paradox in the setting of a field $\KK$ that is complete with respect to a discrete non-Archimedean valuation (e.g., when $\KK$ is the field $\QQ_p$ of $p$-adic numbers for a prime $p$). Namely, the field $\KK$, as well as all balls and spheres in $\KK$, admit a paradoxical decomposition with respect to the isometry group of $\KK$. Such decompositions can be realized using pieces with the Baire property if $\KK$ is separable. Under the additional assumption of local compactness of $\KK$ (e.g., when $\KK=\QQ_p$), any two bounded subsets of $\KK$ with nonempty interiors are equidecomposable with respect to the isometry group of $\KK$.
  Our results complete the study of paradoxical decompositions in the non-Archimedean setting, addressing the one-dimensional case and building on earlier work for higher-dimensional normed spaces over $\KK$ with respect to groups of affine isometries.

  \medskip

  {\footnotesize

  \smallskip

  }
\end{abstract}

\section{Introduction}

In their seminal paper \cite{BT}, Banach and Tarski proved two related results that are among the most surprising and counterintuitive theorems in mathematics. One asserts that
any ball in the Euclidean space $\RR^3$ can be partitioned into finitely many pieces that, after applying isometries of $\RR^3$, reassemble into two balls identical to the original one \cite[Lemma 21]{BT}. In modern terms, any ball in  $\RR^3$ is {\em $G$-paradoxical}, where $G$ denotes the isometry group of $\RR^3$. This phenomenon is known as the Banach--Tarski paradox.
The other result \cite[Theorem 24]{BT} (the ``strong'' paradox) states:

Any two bounded subsets $A$ and $B$ of $\RR^3$ with nonempty interiors are {\em $G$-equidecomposable}, i.e., they can be partitioned into
a finite number of pieces $A_1, \dots , A_n$ and $B_1, \dots ,B_n$, respectively, so that there exist isometries $g_1,\dots,g_n$ of $\RR^3$ with $g_k(A_k)=B_k$ for all $1\le k\le n$.

It is remarkable that these classical paradoxes are not true in the Euclidean spaces of dimension two or one.

\bigskip

The work of Banach and Tarski has inspired numerous refinements and generalizations. They include minimizing the number of pieces needed for a paradoxical decomposition (five is minimal for a ball in $\RR^3$), generalizations for balls and spheres in $\RR^n$, $n\ge 3$, as well as similar results in hyperbolic spaces. The monograph \cite{W} is an excellent source covering most of these topics.
Recently, Tomkowicz proved a general result that unifies several known paradoxes (in complete manifolds) and gives a natural setting for obtaining new ones \cite[Theorem 1.1]{Tomkowicz2017}.

Another area of research involves strengthening the Banach--Tarski paradox by obtaining results on paradoxical decompositions (and equidecomposability) using pieces with special (topological or measure-theoretic) properties. In 1930, Marczewski asked whether a ball in $\RR^3$ admits a paradoxical decomposition using pieces with the Baire property. This question was answered positively by Dougherty and Foreman \cite[Theorem 2.6]{DF}. 
More recently, Grabowski, M\'ath\'e, and Pikhurko established that, for certain group actions, any two measurable relatively compact sets of the same measure and with nonempty interiors are equidecomposable using measurable pieces \cite[Corollary 1.12]{GrMaPikhurko}. This result applies, for instance, to the action of the isometry group of $\RR^n$, $n\ge 3$, cf. \cite[Example 1.9]{GrMaPikhurko}.

\bigskip

In the recent papers \cite{KO-RACSAM} and \cite{KO-BBMS}, we investigated the existence of paradoxical decompositions in normed spaces $(\KK^n, \lVert\blank \rVert)$ over non-Archimedean valued fields $(\KK,\lvert \blank \rvert)$. Among other results, we proved that, for any nontrivial non-Archimedean valuation on $\KK$ and $n\ge2$, the normed space $(\KK^n,\lVert\blank\rVert_{\infty})$, as well as all balls and spheres in $\KK^n$, admit a paradoxical decomposition with respect to the group of affine isometries of $\KK^n$ using four pieces \cite[Theorem 3.10]{KO-BBMS}.

The situation for dimension $n=1$ is different.
Since the affine group $\GA{(1,\KK)}$ is amenable, $\KK$ is not $\GA{(1,\KK)}$-paradoxical.
Yet, there are some bounded subsets of $\KK$ that are paradoxical with respect to the group of affine isometries of $\KK$. We proved their existence in \cite[Theorem 5.10]{KO-RACSAM}, inspired by the paper \cite{TomkowiczBoundedSets}, which developed a theory of bounded paradoxical sets under actions of amenable but not supramenable groups.

It is known that the isometry group $\Iso{\KK}$ of a non-Archimedean valued field $\KK$ is far larger than the group of its affine isometries; see for instance \cite[Theorem 1.2]{Bishop} and \cite[Proposition 1]{Kubzdela}. Therefore, it is reasonable to ask whether some results on paradoxical decompositions and equidecomposability hold in $\KK$ with respect to this larger group $\Iso{\KK}$.

In this paper, we consider a complete discretely valued field $(\KK,\lvert\blank\rvert)$. Any such field is isometric to the field $\FF((T))$ of formal Laurent series over some field $\FF$ (Corollary \ref{cor:important}), provided that $\FF((T))$ is equipped with the valuation \eqref{eq:abs_val_Laurent}, defined in 2.2. Hence, without loss of generality, we may assume that $\KK=\FF((T))$. We will construct isometries witnessing a paradoxical decomposition of $\KK$ (or a ball, or sphere in $\KK$) via conjugation of suitable affine isometries of $(\KK^2,\lVert \blank\rVert_{\infty})$ by a certain additive bijection between $\KK$ and $\KK^2$.

\bigskip

Let us list our main theorems. Throughout this paper, by an {\em isometry} between metric spaces we mean a distance-preserving surjection. We assume below that $\KK$ is a complete discretely valued field and $\Iso{\KK}$ is its isometry group. 

\subsubsection*{Main results}

\begin{enumerate}
  \item The field $\KK$, all balls, and all spheres in $\KK$ are $\Iso{\KK}$-paradoxical using four pieces (Theorem \ref{th:main}).
  \item If $\KK$ is separable, then $\KK$, all balls, and all spheres in $\KK$ are $\Iso{\KK}$-paradoxical using six pieces with the Baire property (Theorem \ref{th:Baire_paradox}).
  \item If $\KK$ is locally compact, then any two bounded sets $A,B\subseteq \KK$ with nonempty interiors are $\Iso{\KK}$-equidecomposable.
  Moreover, if $A,B$ have additionally the Baire property, then they are $\Iso{\KK}$-equidecomposable using pieces with the Baire property (Theorem \ref{thm:strong_paradox}).
\end{enumerate}

In particular, all of the results stated above hold for the valued field $(\mathbb{Q}_p,\lvert\blank\rvert_p)$ of $p$-adic numbers, for any prime number $p$.

\section{Preliminaries}

\subsection{Non-Archimedean valued fields and normed spaces}

A {\em valuation} (or {\em absolute value}) on a field $\KK$ is a function $\lvert\blank\rvert\colon \KK\to[0,\infty)$ satisfying, for all $x,y\in \KK$, the following conditions:

\begin{enumerate*}[label=(\arabic*)]
    \item $\lvert x \rvert=0 \Leftrightarrow x=0$,
    \item $\lvert xy \rvert=\lvert x \rvert\lvert y \rvert$,
    \item $\lvert x+y \rvert\leq \lvert x \rvert+\lvert y \rvert$.
\end{enumerate*}\\
The pair $(\KK,\lvert\blank\rvert)$, written simply as $\KK$ if the valuation is known from the context, is then called a {\em valued field}.
If we replace condition (3) with the stronger one (3'): \(\lvert x+y\rvert\leq\max\{\lvert x\rvert,\lvert y\rvert\}\), then the valuation \(\lvert\blank\rvert\) is called {\em non-Archimedean}.

If $(X,d)$ is a metric space, $x_0\in X$ and $r>0$, we put:
\vspace{-0.5\baselineskip}
\[
\begin{aligned}
B[x_0,r] &:= \{x\in X: d(x,x_0)\leq r\}, \\
B(x_0,r) &:= \{x\in X: d(x,x_0)<r\}, \\
S[x_0,r] &:= \{x\in X: d(x,x_0)= r\}.
\end{aligned}
\]
\vspace{-0.5\baselineskip}

A metric $d$ on a set $X$ is called an {\em ultrametric} if it satisfies the {\em strong triangle inequality}:
$d(x,z) \leq \max\{d(x,y), d(y,z)\}$ for all $x,y,z \in X$. The pair $(X,d)$ is then called an {\em ultrametric space}. Any ultrametric space $(X,d)$ has the {\em isosceles property}: If $x,y,z \in X$ and $d(x,y)\neq d(y,z)$, then
$d(x,z) = \max\{d(x,y), d(y,z)\}$. Note that, in an ultrametric space, any point belonging to a ball (open or closed) of a given radius may play the role of its center. Moreover, each of the sets $B[x_0,r]$, $B(x_0,r)$, $S[x,r]$ is closed and open in the topology of $(X,d)$.

Each valuation $\lvert\blank\rvert$ on a field $\KK$ induces a metric $d\colon \KK\times \KK \to [0, \infty)$ by $d(x,y):=\lvert x-y \rvert$. If $\lvert\blank\rvert$ is non-Archimedean, $d$ is an ultrametric. We say that $(\KK,\lvert\blank\rvert)$ is {\em complete} if $(\KK,d)$ is a complete metric space.

Let $(\KK,\lvert\blank\rvert)$ be a fixed non-Archimedean valued field. We introduce the following notation:
$D_{\KK}= B[0,1]$, $\mathfrak{m}_{\KK}= B(0,1)$, and $D_{\KK}^{\times} = S[0,1].$
The ball $D_{\KK}$ is a subring of $\KK$ called the {\em ring of integers} (or {\em valuation ring}) of $\KK$; $\mathfrak{m}_{\KK}$ is the unique maximal ideal of $D_{\KK}$ and $D_{\KK}^{\times}$ is the group of units (i.e., invertible elements) of $D_{\KK}$ (cf. \cite[Lemma 1.2]{Schneider}). The quotient ring $D_{\KK}/\mathfrak{m}_{\KK}$ is a field called the {\em residue field} of $(\KK,\lvert\blank\rvert)$.

We denote by $\KK^{\times}$ the multiplicative group of $\KK$. The set $\lvert \KK^{\times} \rvert:=\{\lvert x \rvert\colon x\in \KK^{\times}\}$ is a subgroup of the multiplicative group $\RR_{+}$ and is called the {\em value group} of $\KK$.
If $\lvert \KK^{\times} \rvert=\{1\}$, the valuation $\lvert\blank\rvert$ is called {\em trivial}. If the set $\lvert\KK^{\times}\rvert\cap(0,1)$ has a greatest element $c$, then
$\lvert \KK^{\times} \rvert=\{c^n: n\in \mathbb{Z}\}$ and the valuation is called {\em discrete}. In the remaining case, $\lvert \KK^{\times} \rvert$ is a dense subset of $[0,\infty)$ and the valuation $\lvert\blank\rvert$ is called {\em dense}.

Observe that if $\lvert \KK^{\times}\rvert=\{c^n: n\in\ZZ\}$, where $c\in (0,1)$, then {$B(x_0, c^k)=B[x_0,c^{k+1}]$} for all $x_0\in\KK$, $k\in \ZZ$. Therefore, in the case of a discretely valued $\KK$, every ball is a closed one.

It is known that a non-Archimedean nontrivially valued field $(\mathbb{K}, \lvert\blank\rvert)$ is locally compact if and only if it is complete, the valuation $\lvert\blank\rvert$ is discrete, and the residue field $D_{\KK}/\mathfrak{m}_{\KK}$ is finite {\cite [Corollary 12.2]{UC}}.

A standard example of a non-Archimedean locally compact field is the field $(\QQ_p,\lvert\blank\rvert_p)$ of {\em $p$-adic numbers} for any prime $p$. It is defined as the completion of $\QQ$ with respect to the {\em $p$-adic valuation}, i.e.:
\[
    \lvert 0 \rvert_p:= 0, \; \left\lvert p^k \frac{m}{n}\right \rvert_p:=p^{-k} \quad \text{for } k,m,n\in\mathbb{Z}, \quad p\nmid m, p\nmid n.
\]

From now on, we will mostly consider fields $\KK$ that are complete with respect to a discrete non-Archimedean valuation $\lvert\blank\rvert$. Such a $\KK$ will be briefly called a {\em complete discretely valued field}. Sometimes it is useful to distinguish an element $\pi\in\KK$ satisfying $\lvert\pi\rvert=c$, where $c\in(0,1)$ is the generator of $\lvert \KK^{\times}\rvert$; such a fixed $\pi$ is called a {\em uniformizer}.

\bigskip

Let $X$ be a linear space over a non-Archimedean valued field $(\KK,\lvert\blank\rvert)$. A function $\lVert \blank \rVert \colon X\to [0, \infty)$ is called a {\em non-Archimedean norm} (briefly: norm) on $X$ if, for all $x,y \in X$ and $\alpha\in \KK$, we have the following conditions:

\begin{enumerate*}[label=(\arabic*)]
   \item $\lVert x \rVert =0 \Leftrightarrow x=0$,
   \item $\lVert \alpha x \rVert = \lvert \alpha \rvert \lVert x \rVert $,
   \item $\lVert x+y \rVert  \leq \max\{\lVert x \rVert ,\lVert y \rVert \}$.
\end{enumerate*}\\
The pair $(X,\lVert \blank \rVert )$ is then called a {\em non-Archimedean normed space} over $\KK$. If additionally $\left\{\lVert x\rVert\colon x\in X\setminus\{0\}\right\}=\lvert\KK^{\times}\rvert$, the normed space $(X,\lVert\blank\rVert)$ is called {\em strongly solid}.
Any non-Archimedean norm $\lVert \blank \rVert $ induces an ultrametric $d$ on $X$, namely $d(x,y):=\lVert x-y \rVert $. If $X$ is complete with respect to $d$, we call $X$ a {\em non-Archimedean Banach space}.

We note an important property of non-Archimedean Banach spaces. Namely, a series $\sum_{n=1}^{\infty} x_n$, where $x_n\in X$, is convergent if and only if $x_n \to 0$, see \cite[Theorem 2.1.3]{PG}.

If $g\colon X\to X$ is an affine mapping $g(x):=L(x)+\tau$ with linear part $L\colon X\to X$ and translation part $\tau\in X$, we will write $g=(L,\tau)$. Clearly, $g$ is an isometry of $(X,\lVert\blank\rVert)$ if and only if $L$ is.

\bigskip

Let $(\KK,\lvert\blank\rvert)$ be a non-Archimedean valued field and $n\ge 1$. Define
\[
\lVert \blank \rVert_{\infty}\colon \KK^n \to [0,\infty), \quad (x_1,\dots,x_n)\mapsto \max\{\lvert x_1\rvert, \dots, \lvert x_n\rvert\}.
\]
The function $\lVert\blank\rVert_{\infty}$ is a non-Archimedean norm on $\KK^n$ and is called the {\em maximum norm}. The normed space $(\KK^n,\lVert\blank \rVert_{\infty})$ is obviously strongly solid. It is a Banach space if $\KK$ is complete.

In analogy with the Euclidean norm on $\mathbb{R}^n$, one can define $\lVert \blank \rVert_{2}\colon \KK^n \to [0,\infty)$ by $(x_1,\dots,x_n)\mapsto (\lvert x_1\rvert^2 + \dots + \lvert x_n\rvert^2)^{1/2}$. However, this function is not a non-Archimedean norm on $\KK^n$ since it does not satisfy the defining condition (3)---the induced metric is not an ultrametric.

\medskip

According to the characterization of linear isometries of $(\KK^n,\lVert\blank\rVert_{\infty})$ stated in \cite[Theorem 2.2]{KO-RACSAM} (cf. \cite[Corollary 1.7]{vRNotes}), a linear mapping $L\colon \KK^2 \to \KK^2$ with matrix
$L=\left(\begin{smallmatrix}
  \alpha & \beta\\
  \gamma & \delta
\end{smallmatrix}\right)$ in the canonical basis of $\KK^2$
is an isometry of $(\KK^2, \lVert\blank\rVert_{\infty})$ if and only if $L\in\GL(2,D_{\KK})$, that is, $\lvert\det{L}\rvert=1$ and $\lvert\alpha\rvert,\lvert\beta\rvert,\lvert\gamma\rvert,\lvert\delta\rvert\le 1$.
Let $\SL(2,\KK):=\{A \in M_2(\KK): \det{A}=1\}$ denote the special linear group over $\KK$.

In \cite[Definition 2.1]{KO-BBMS}, we introduced certain specific families of groups of linear and affine transformations of $\KK^2$. We recall two of them, which will be useful later. Namely, for $\varepsilon\in (0,1)$, let:
\begin{align}
\SL(2,D_{\KK},\varepsilon)&:=\left\{\begin{pmatrix}
  \alpha & \beta\\
  \gamma & \delta
\end{pmatrix}
\in \SL(2,\KK): \lvert \alpha -1\rvert, \lvert \delta -1\rvert, \lvert \beta\rvert, \lvert \gamma\rvert \le \varepsilon\right\},\label{eq:def_of_SL(2,D_K,e)}\\
\SA(2,D_{\KK},\varepsilon)&:=\{(L,\tau)\in \SL(2,D_{\KK},\varepsilon)\times \KK^{2}: \quad \lVert\tau\rVert_{\infty} \le \varepsilon\}.\nonumber
\end{align}

Notice that the conditions $\lvert \alpha - 1\rvert, \lvert \delta-1\rvert \le \varepsilon < 1$ imply $\lvert\alpha\rvert=\lvert\delta\rvert =1$, so $\SL(2,D_\KK,\varepsilon)$ is a subgroup of $\GL(2,D_{\KK})$ for any $\varepsilon\in (0,1)$. Therefore, all elements of $\SL(2,D_{\KK},\varepsilon)$ and $\SA(2,D_{\KK},\varepsilon)$ are isometries of $(\KK^2,\lVert\blank\rVert_{\infty})$.

\subsection{Laurent sequences}

Let $R$ be a set of cardinality at least two with a distinguished element $0\in R$. By a {\em Laurent sequence} over $R$ we mean any function $\varphi \colon \ZZ \to R$ such that its {\em support} $\supp{\varphi}:=\{n\in \ZZ \colon \varphi(n)\ne 0\}$ is well-ordered.
We denote by $\LS{(R)}$ the set of all Laurent sequences over $R$.

For any number $c\in (0,1)$, the formula:
\[
d_c(\varphi,\psi):=\begin{cases}
                        c^{\min\{n\in \ZZ \colon \varphi(n)\ne \psi(n)\}} & \text{if } \varphi\ne \psi,\\
                        0 & \text{if } \varphi= \psi
                    \end{cases}
\]
defines an ultrametric on $\LS{(R)}$. We denote by $\LS{(R,c)}$ the resulting ultrametric space.
Note that if $R_1$ and $R_2$ are sets of equal cardinality and $c_1=c_2$, then the ultrametric spaces $\LS{(R_1,c_1)}$ and $\LS{(R_2,c_2)}$ are isometric.

If $R=\FF$ is a field, then a Laurent sequence over $\FF$ is called a {\em (formal) Laurent series} over $\FF$, and we write $\FF((T))$ for $\LS{(\FF)}$. A Laurent series $\varphi \in \FF((T))$ is traditionally written as the formal sum 
\begin{equation}\label{eq:formal_sum}
    \varphi = \sum_{n=-\infty}^{\infty} \varphi(n) T^n,
\end{equation}
 with an indeterminate symbol $T$.
Addition and multiplication are defined in the usual way in $\FF((T))$, see \cite[Definition 7.1]{FPS-survey}. Thus $\FF((T))$ becomes a field \cite[Theorem 7.2]{FPS-survey}, into which $\FF$ is canonically embedded.  For a given $c\in (0,1)$, we introduce the following non-Archimedean valuation on $\FF((T))$:
\begin{equation}\label{eq:abs_val_Laurent}
  \left\lvert \sum_{n=-\infty}^{\infty} \varphi(n) T^n
 \right\rvert:= c^{\min \supp{\varphi}} \, \text{ if } \varphi(k)\ne 0, \quad \lvert 0\rvert:=0.  
\end{equation}
In this way, $\FF((T))$ becomes a complete discretely valued field with the valuation ring $\FF[[T]]:=\{\varphi\in \FF((T))\colon \supp{\varphi}\subseteq \mathbb{N}_0\}$ and the residue field isomorphic to $\FF$, cf. \cite[Example 2 on p. 6]{Serre-LF}. The metric induced by the valuation \eqref{eq:abs_val_Laurent} is precisely the one of $\LS{(\FF,c)}$.

\subsection{Group actions and paradoxical decompositions}

By a {\em group action} $G\acts X$ of a group $G$ on a set $X$ we mean a function $G\times X \ni (g,x)\mapsto g.x\in X$ such that $(gh).x=g.(h.x)$ and $1.x=x$ for all $g,h\in G$,  $x\in X$. An {\em isomorphism} of group actions $G\acts X$ and $H \acts Y$ is a pair $(f,\Phi)$ consisting of a bijection $f\colon X \to Y$ and a group isomorphism $\Phi\colon G\to H$ such that
$\Phi(g).f(x)=f(g.x)$ for all $g\in G$, $x\in X$.

Let $X,Y$ be any sets and $G$ a group of bijections of $X$. For any bijection $f\colon X\to Y$, let $fGf^{-1}:=\{f\circ g\circ f^{-1}\colon g\in G\}$, which is a group of bijections of $Y$. Define $G \ni g \mapsto \Phi(g):=f\circ g \circ f^{-1}$. Then $(f,\Phi)$ is clearly an isomorphism of group actions $G\acts X$ and $fGf^{-1}\acts Y$, we call it the isomorphism {\em induced} by $f$.
In particular, if $X,Y$ are metric spaces and $\Iso{X},\Iso{Y}$ are their (full) isometry groups, then every isometry $f\colon X\to Y$ induces an isomorphism between the group actions $\Iso{X}\acts X$ and $\Iso{Y}\acts Y$.

A group action $G\acts X$ is called {\em locally commutative} if, for all $x\in X$, the {\em stabilizer} $\Stab_G{x}:=\{g\in G: g.x=x\}$ of $x$ is an Abelian subgroup of $G$. The action is {\em free} if each of the stabilizers is trivial. 
Suppose that $(f,\Phi)$ is an isomorphism of group actions $G\acts X$ and $H\acts Y$. Since ${\Phi(\Stab_G{x})=\Stab_H{f(x)}}$ for all $x\in X$, the properties of local commutativity and freeness are preserved by isomorphisms of group actions.

\bigskip

Let $G\acts X$ be a group action. We say that a subset $E\subseteq X$ is {\em $G$-paradoxical using $r$ pieces} \cite[Definition 5.1]{W} if, for some positive integers $m$, $n$ with $m+n=r$, there exist subsets $A_1,\dots,A_m,B_1,\dots,B_n$ of $E$ and elements $g_1,\dots,g_m,h_1,\dots,h_n\in G$ such that
\begin{equation}\label{eq:rparwzor}
    E=A_1 \sqcup \dots \sqcup A_m \sqcup B_1 \sqcup \dots \sqcup B_n =\bigsqcup_{i=1}^{m} g_i.A_i = \bigsqcup_{j=1}^{n} h_j.B_j,
\end{equation}
where the symbols $\sqcup$ and $\bigsqcup$ indicate that the components of the respective unions are pairwise disjoint.
In such a case the decomposition \eqref{eq:rparwzor} of $E$ is called {\em a $G$-paradoxical decomposition}. The empty set is obviously $G$-paradoxical (using two pieces).
In general, we do not require $E$ to be $G$-invariant; consider for example the classical Banach--Tarski paradox, where the unit ball in $\RR^3$ is not invariant under the isometry group of $\RR^3$ (the invariance is violated by translations).

We say that two sets $A, B\subseteq X$ are {\em $G$-equidecomposable} \cite[Def. 3.4]{W} if, for some $n\in\mathbb{N}$, there exist partitions
\begin{equation*}
        A=\bigsqcup_{i=1}^n A_i, \quad B=\bigsqcup_{i=1}^n B_i,
\end{equation*}
and elements $g_1,\dots,g_n\in G$ such that $g_i.A_i=B_i$ for all $1\leq i\leq n.$ Then we write $A\sim_{G} B$, and thus obtain an equivalence relation $\sim_{G}$ on the power set of $X$.

\begin{rem}\label{rem:isomorphic_actions}
If $(f,\Phi)$ is an isomorphism of group actions $G\acts X$ and $H\acts Y$, then any $E\subseteq X$ is $G$-paradoxical using $r$ pieces if and only if $f(E)\subseteq Y$ is $H$-paradoxical using $r$ pieces. Moreover, for any $A,B\subseteq X$, the condition $A\sim_{G}B$ is equivalent to $f(A)\sim_{H}f(B)$.
\end{rem}

It is known that, for a group action $G\acts X$, the set $X$ is $G$-paradoxical using four pieces if and only if $G$ contains a free subgroup $F_2$ of rank two such that the restricted action $F_2 \acts X$ is locally commutative. More precisely, $X$ has a decomposition
$X=A_1 \sqcup A_2 \sqcup A_3 \sqcup A_4 = A_1 \sqcup a.A_2 = A_3 \sqcup b.A_4$ for $a,b \in G$ if and only if $\{a,b\}$ is a basis of a free subgroup of rank two $F_2\le G$ whose action on $X$ is locally commutative \cite[Theorems 5.5 and 5.8]{W}.

The following result will be used in the proof of Theorem \ref{thm:strong_paradox}:\\
If $E\subseteq X$ is $G$-paradoxical, then $E\sim_{G}A$ for any $A\subseteq E$ such that
$E\subseteq g_1 A \cup \dots \cup g_n A$ for some $n\in\mathbb{N}$, $g_1,\dots,g_n \in G$.
Hence, any two such subsets $A', A''$ of $E$ are $G$-equidecomposable and $G$-paradoxical \cite[Corollary 10.22]{W}.

\subsection{Baire paradoxical decompositions}

We recall some notions from descriptive set theory.
Let $X$ be a topological space.
A subset $A\subseteq X$ has the {\em Baire property} (in $X$) if there exists an open subset $U$ of $X$ such that the symmetric difference $A\,\triangle\, U$ is a {\em meager} set in $X$, that is, a countable union of nowhere dense sets in $X$.
The family $\mathcal{B}(X)$ of all subsets of $X$ with the Baire property is a $\sigma$-algebra \cite[Proposition 8.22]{Kechris}.
Moreover, if $E\in\mathcal{B}(X)$ and $A\in \mathcal{B}(E)$, then $A\in \mathcal{B}(X)$; see \cite[\S~11.V.1,~p.~60]{Kuratowski}.

Let $E\in \mathcal{B}(X)$ and let $G$ act on $X$ by homeomorphisms. We say that $E$ is {\em Baire $G$-paradoxical} using $r$ pieces if $E$ is $G$-paradoxical using $r$ pieces and all of the pieces involved in the decomposition \eqref{eq:rparwzor} have the Baire property in $X$. 

Dougherty and Foreman provided a sufficient condition under which a {\em Polish space} (i.e., a separable completely metrizable topological space) $X$ is Baire $G$-paradoxical using six pieces, for a group $G$ acting on $X$ by homeomorphisms. Namely, it suffices that the action $G \acts X$ is locally commutative, free on a {\em comeager} set (i.e., on the complement of a meager set), and $G$ contains a free subgroup of rank two \cite[Theorem 5.9]{DF}.
More recently, Marks and Unger proved a general theorem that establishes the existence of Baire $G$-paradoxical decomposition, provided that $G$ acts on $X$ by Borel automorphisms and $X$ admits a usual ``set-theoretical'' $G$-paradoxical decomposition \cite[Theorem~1.1]{MU}. However, the cited result does not give an explicit upper bound of the number of pieces. We also mention that Wehrung proved that six is the minimal number of pieces needed for a Baire paradoxical decomposition of a compact metric space with respect to isometries \cite[Corollary]{Weh}.

\bigskip

In the proof of Theorem \ref{thm:strong_paradox}, we will use a criterion from \cite{GrMaPikhurko}, which contains powerful results on equidecomposability using pieces with some special properties (Borel, measurable, Baire---among others). The criterion can be adapted to our setting as follows: 

Let $G\acts X$ be a group action on a locally compact Polish space $X$ by homeomorphisms. Assume that any two relatively compact subsets of $X$ with nonempty interiors are $G$-equidecomposable.
Then, any two relatively compact sets $A,B\in \mathcal{B}(X)$ with nonempty interiors are Baire $G$-equidecomposable.

The statement is essentially a simplified version of \cite[Proposition 1.8\,(ii)]{GrMaPikhurko} for homeomorphic actions.

\section{Auxiliary results}

We begin with a proposition on the series expansion of elements of a strongly solid Banach space $X$ over a complete discretely valued field $\KK$. In the special case $X=\KK$, the result is well-known \cite[Theorem 12.1]{UC}.

\begin{prop}\label{prop:general_series_expansion}
Let $(X,\lVert\blank\rVert)$ be a strongly solid Banach space over a complete discretely valued field $(\KK,\lvert\blank\rvert)$ whose value group is generated by $c\in (0,1)$.
Let $\pi\in \KK$ with $\lvert \pi\rvert = c$, and let $R$ be a fixed set of representatives of the cosets of $B_{X}[0,1]$ modulo the subgroup $B_{X}(0,1)$; assume that $R$ contains the zero vector of $X$.

Then, for each $x\in X$, there exists a unique Laurent sequence $\varphi\in \LS{(R)}$ such that
\begin{equation}\label{eq:series_expansion_Banach_space}
  x = \sum_{n=-\infty}^{\infty}  \pi^n \varphi(n).  
\end{equation}
Moreover, the above correspondence $x \mapsto \varphi$ induces an isometry from $(X,\lVert\blank\rVert)$ onto $\LS{(R,c)}$.
\end{prop}

\noindent {\it Remark.} The summation symbol in \eqref{eq:series_expansion_Banach_space} denotes the limit of partial sums in $X$, taken with respect to the topology induced by $\lVert \blank \rVert$. Notice that the series is in fact one-sided since $\varphi(n)=0$ for sufficiently small $n$.

\begin{proof}
Let $x\in X$, $x\ne 0$.
We will proceed as in the proof of \cite[Theorem~12.1]{UC} by building recursively an appropriate Laurent sequence $\varphi\in \LS{(R)}$.

Since $(X,\lVert\blank\rVert)$ is strongly solid, there exists $m\in \ZZ$ such that $\lVert x \rVert=\lvert \pi^m \rvert$. We put $\varphi(n):=0$ for all $n<m$.
Let $\varphi(m)$ be the unique element of $R$ satisfying $\lVert \pi^{-m}x - \varphi(m)\rVert < 1$.
By the assumptions, $B_{X}(0,1)=\pi B_{X}[0,1]$, so there is a unique $\varphi(m+1)\in R$ with
${\lVert \pi^{-m-1}x - \pi^{-1}\varphi(m) - \varphi(m+1)\rVert < 1}$.
Continuing recursively, we obtain $\varphi\in \LS{(R)}$ such that
\[\left\lVert \pi^{-m-n}x - \sum_{k=0}^{n}\pi^{-n+k} \varphi(m+k)\right\rVert < 1 \quad \text{for all } n\ge 0.\]
Multiplying both sides of the last inequality by $\lvert \pi^{m+n}\rvert$ and passing to the limit as $n\to\infty$, we get
$x = \sum_{k=0}^{\infty}\pi^{m+k} \varphi(m+k)$, hence \eqref{eq:series_expansion_Banach_space}.

We will simultaneously prove that $\varphi$ is unique and the mapping $X\ni x\mapsto \varphi \in\LS{(R,c)}$ preserves distance.
Assume that $x, y\in X$ have expansions
$x = \sum_{n=-\infty}^{\infty}  \pi^n \varphi(n)$ and $y= \sum_{n=-\infty}^{\infty}  \pi^n \psi(n)$ for some $\varphi,\psi\in\LS{(R)}$, $\varphi\ne \psi$. Let $k:=\min\{n\in\ZZ \colon \varphi(n)\ne \psi(n)\}$, so that $d_c(\varphi,\psi)=c^k$. We have
\[x-y = \pi^k (\varphi(k)-\psi(k)) + \sum_{n=k+1}^{\infty} \pi^n (\varphi(n)-\psi(n)).\]
The elements $\varphi(k)$, $\psi(k)$ represent different cosets of $B_{X}[0,1]$ modulo $B_{X}(0,1)$, so $\lVert \varphi(k)-\psi(k)\rVert = 1$.
Since $\lVert \sum_{n=k+1}^{\infty} \pi^n (\varphi(n)-\psi(n)) \rVert \le c^{k+1} < c^{k} = \lVert \pi^k (\varphi(k)-\psi(k))\rVert$, we obtain $\lVert x-y\rVert=c^k=d_c(\varphi,\psi)$, as desired. In particular, $\lVert x- y\rVert >0$, so the expansion of each $x\in X$ is indeed unique.

Finally, since $X$ is complete, the sequence in \eqref{eq:series_expansion_Banach_space} converges for each $\varphi\in \LS{(R)}$. Therefore, the mapping $x\mapsto \varphi$ is onto $\LS{(R)}$. 
\end{proof}

Let us consider the case $X=\KK=\FF((T))$ with the valuation \eqref{eq:abs_val_Laurent}.
If we interpret $T$ as the indicator function of $\{1\}$, then, after putting $\pi:=T$ and $R:=\FF$, the expansion \eqref{eq:series_expansion_Banach_space} of $x:=\varphi$ looks exactly like \eqref{eq:formal_sum}. Therefore, the notation \eqref{eq:formal_sum} is meaningful in the analytic sense (as the limit of the partial sums), as well.

According to the following corollary, we can think of $\FF((T))$ as our model example of a complete discretely valued field with the residue field $\FF$ and the value group generated by $c$.

\begin{cor}\label{cor:important}
    Every complete discretely valued field $\KK$ with residue field $\FF$ and value group generated by $c$, $c\in (0,1)$, is isometric to $\FF((T))=\LS{(\FF,c)}$.
\end{cor}

\begin{proof}    
An application of Proposition \ref{prop:general_series_expansion} yields an isometry $\KK \to \LS{(R,c)}$, for a set $R$ of representatives of the cosets of $D_{\KK}$ modulo $\mathfrak{m}_{\KK}$ with $0\in R$. Since by definition $\FF$ has the same cardinality as $R$, the spaces $\LS{(R,c)}$ and $\LS{(\FF,c)}$ are also isometric.
\end{proof}

\noindent\textit{Remark.} In a special case, Corollary \ref{cor:important} can be strengthened as follows. If additionally $\ch\KK=\ch\FF$, then $\KK$ is isometrically isomorphic to $\FF((T))$, see \cite[Ch.~II, \S4, Theorem~2]{Serre-LF}. In the mixed-characteristic case, one obtains only an isometry of metric spaces, not an isomorphism of valued fields.

\bigskip

Let $\FF$ be any field and $c\in (0,1)$. Throughout the rest of this section, we consider the field $\KK:=\FF((T))$ with the valuation defined by \eqref{eq:abs_val_Laurent}. 
The element $\pi:=T$ is a natural uniformizer in $\KK$.

Let us endow the linear space $X:=\KK^2$ over $\KK$ with the norm $\lVert \blank\rVert_{\infty}$. We choose $R:=\FF^{2}$ as a set of representatives of the cosets of $B_{X}[0,1]$ modulo the subgroup $B_{X}(0,1)$.

For every $n\in \ZZ$, we define the $n$-th {\em projection operator} $\pr_n \colon \KK^2 \to \FF^2$ as follows:
\[\pr_n{x}:=\varphi(n),\]
where $\varphi\in \LS{(R)}$ is the Laurent sequence satisfying \eqref{eq:series_expansion_Banach_space} for $x\in X$ under our choices of $\pi$ and $R$. Every projection operator is clearly $\FF$-linear. Moreover, it has the following property.

\begin{lem}\label{lem:continuity_discrete_topology}
For each $n\in\ZZ$, the operator $\pr_n\colon \KK^2\to \FF^2$ is continuous with respect to the $\lVert\blank\rVert_{\infty}$-norm topology on $\KK^2$ and the discrete topology on $\FF^2$.
\end{lem}

\begin{proof}
Fix $n\in \ZZ$. It is sufficient to observe that $\pr_n{x}=\pr_n{y}$ for all $x,y\in \KK^2$ with $\lVert x-y\rVert_{\infty} < c^{n}$.
\end{proof}

\bigskip

We are going to describe how the $n$-th projection behaves after composition with transformations from $\SL(2,D_{\KK},\varepsilon)$, $\varepsilon\in(0,1)$. First, observe that, if $L\colon \KK^2 \to \KK^2$ is a linear isometry, we have $\lVert L(r)\rVert_{\infty}=\lVert r\rVert_{\infty}\in \{0,1\}$ for all $r\in \FF^2$. Thus, all of the coefficients next to negative powers of $T$ in the series expansion of $L(r)$ are equal to $(0,0)$.

\begin{lem}\label{lem:preservation_of_projections}
    Let $L\in \SL(2,D_{\KK},\varepsilon)$ for some $\varepsilon\in (0,1)$, and let $x\in\KK^2$, $n\in\ZZ$ be such that $\lVert x\rVert_{\infty}=c^{n}$. Then $\pr_{n}{L(x)}=\pr_{n}{x}$, and $\pr_{k}{L(x)}=(0,0)$ for all $k<n$.
\end{lem}

\begin{proof}
  Let $L$, $x$ and $n$ be as in the assumption.
  First, we will show that $\pr_{0}L(r)=r$ for all $r\in \FF^2$.
  Since the projection operator is $\FF$-linear, it is sufficient to verify that the composition $\pr_0 \circ L$ fixes the (canonical) basis vectors of $\FF^2$.

By \eqref{eq:def_of_SL(2,D_K,e)}, we can write
\[
L =\begin{pmatrix}
  1 + \sum_{k=1}^{\infty} \alpha_k T^k & \sum_{k=1}^{\infty} \beta_k T^k\\
  \sum_{k=1}^{\infty} \gamma_k T^k & 1 + \sum_{k=1}^{\infty} \delta_k T^k
\end{pmatrix},
\]
where $(\alpha_k)_{k=1}^{\infty}$, $(\beta_k)_{k=1}^{\infty}$, $(\gamma_k)_{k=1}^{\infty}$ and  $(\delta_k)_{k=1}^{\infty}$ are some sequences of elements of $\FF$. Hence,
\begin{gather*}
L(1,0)=\left(1 + \sum_{k=1}^{\infty} \alpha_k T^k, \sum_{k=1}^{\infty} \gamma_k T^k\right)=(1,0) + \sum_{k=1}^{\infty} T^k (\alpha_k,\gamma_k),\\
L(0,1)=\left(\sum_{k=1}^{\infty} \beta_k T^k,1 + \sum_{k=1}^{\infty} \delta_k T^k \right)=(0,1) + \sum_{k=1}^{\infty} T^k (\beta_k,\delta_k).
\end{gather*}
It follows that $\pr_0{L(1,0)}=(1,0)$ and $\pr_0{L(0,1)}=(0,1)$. Hence, $\pr_{0}L(r)=r$ for all $r\in \FF^2$.

\medskip

 Since $L$ is an isometry of $(\KK^2,\lVert\blank\rVert_{\infty})$, we have $\lVert L(x)\rVert_{\infty} = c^n$, so $\min \{k\in\ZZ: \pr_{k}{L(x)}\ne (0,0)\}=n$. Therefore, $\pr_{k}{L(x)}=(0,0)$ for all $k<n$.

Let $\varphi$ be the Laurent sequence over $\FF^2$ representing $x$, and write $x=\sum_{k=n}^{\infty}T^k \varphi(k)$.
  By $\KK$-linearity and continuity of $L$, we obtain $L(x)=\sum_{k=n}^{\infty}T^k L(\varphi(k))$.
  Let us abbreviate $S_l:=\sum_{k=n}^{l}T^k L(\varphi(k))$ for $l\ge n$. By Lemma \ref{lem:continuity_discrete_topology}, $\pr_{n}{L(x)}$ equals the limit of $\pr_{n}{S_l}$, as $l\to \infty$, in the discrete topology on $\FF^2$. The projection operator is additive, so $\pr_{n}{S_l}=\sum_{k=n}^{l} \pr_{n}{(T^k L(\varphi(k)))}=\sum_{k=n}^{l}\pr_{n-k}{L(\varphi(k))}$.
  It follows from the observation preceding this lemma that $\pr_{n-k}{L(\varphi(k))}$ vanishes for all $k>n$. Hence, $\pr_n{S_l}=\pr_0{L(\varphi(n))}$ for all $l\ge n$. In conclusion, 
  \[\pr_{n}{L(x)}=\pr_0{L(\varphi(n))}=\varphi(n)=\pr_{n}{x}.
  \]
\end{proof}

Let us define a function $f\colon \KK \to \KK^2$, motivated by the partition of integers into even and odd ones, as follows:
\begin{equation}\label{eq:def_of_f}
    f\left(\sum_{k=-\infty}^{\infty}x_k T^k\right):=\left(\sum_{k=-\infty}^{\infty}x_{2k}T^k, \sum_{k=-\infty}^{\infty}x_{2k+1}T^k\right).
\end{equation}

Clearly, $f$ is an $\FF$-linear bijection. The value of $f$ at $\sum_{k=-\infty}^{\infty}x_k T^k$ has the series expansion $\sum_{k=-\infty}^{\infty}T^k (x_{2k},x_{2k+1})$; recall that we preserve our choice of $\pi=T$ and $R=\FF^2$. 

Let us compute how the norm of $f(x)$ in $\KK^2$ is related to $\lvert x\rvert$ for $x\in\KK$. We claim that:
\begin{equation}\label{eq:norm_formula}
  \lVert f(x)\rVert_{\infty}=c^{\lfloor m/2\rfloor} \quad \text{for any } x\in\KK \; \text{with }\lvert x\rvert=c^{m}.
\end{equation}
Indeed, if $m=\min \supp x$, then $\min\{k\in \ZZ: (x_{2k},x_{2k+1})\ne (0,0)\}=\lfloor m/2 \rfloor$.
We note an important consequence of \eqref{eq:norm_formula}, namely:
\begin{equation*}
\lvert x\rvert^{1/2}\le \lVert f(x)\rVert_{\infty} \le c^{-1/2}\lvert x\rvert^{1/2} \quad \text{for all } x\in \KK.
\end{equation*}
It follows that $f$ is a homeomorphism and satisfies the inequality:
\begin{equation}\label{eq:estimation_f^-1}
  \lvert f^{-1}(y)\rvert \le \lVert y\rVert_{\infty}^{2} \quad \text{for all } y\in \KK^2.
\end{equation}

\bigskip

The following statement emphasizes the importance of the function $f$. Namely, the conjugation by $f^{-1}$ turns some affine isometries of $\KK^2$ into isometries of $\KK$. 

\begin{lem}\label{lem:f^-1gf_is_isometry}
Let $f\colon \KK\to \KK^2$ be as in \eqref{eq:def_of_f}. If $g\in \SA(2,D_{\KK},\varepsilon)$ for some $\varepsilon\in (0,1)$, then $f^{-1}\circ g \circ f$ is an isometry of $\KK$.
\end{lem}

\begin{proof}
Let $g=(L,\tau)\in \SA(2,D_{\KK},\varepsilon)$.
First, we will show that
\begin{equation}\label{eq:f^-1Lf}
\lvert( f^{-1}\circ L \circ f)(x)\rvert=\lvert x\rvert \quad \text{for all } x\in\KK.
\end{equation}

Fix any $x=\sum_{k=-\infty}^{\infty}x_k T^k\in \KK\setminus\{0\}$ and set $m:=\min \supp{x}$, $n:=\lfloor m/2\rfloor$. Then $\lvert x\rvert = c^{m}$, so $\lVert f(x)\rVert_{\infty}=c^n$ by \eqref{eq:norm_formula}.
It follows from Lemma \ref{lem:preservation_of_projections} that
\[
\pr_{k}{L(f(x))}=\begin{cases}
  \pr_{n}{f(x)}=(x_{2n},x_{2n+1}) & \text{for } k=n,\\
  (0,0) & \text{for } k<n.
                \end{cases}
\]
Hence, we can write
\[
L(f(x))=T^n (x_{2n}, x_{2n+1}) + \sum_{k=n+1}^{\infty} T^k (\alpha_k, \beta_k)
\]
for some sequences $(\alpha_k)_{k=n+1}^{\infty}$, $(\beta_k)_{k=n+1}^{\infty}$ of elements of $\FF$.

Suppose that $z=\sum_{k=-\infty}^{\infty}z_k T^k\in \KK$ is such that $f(z)=L(f(x))$. By \eqref{eq:def_of_f}, we obtain
\begin{gather*}
\sum_{k=-\infty}^{\infty}z_{2k}T^k = x_{2n} T^n +\sum_{k=n+1}^{\infty}\alpha_k T^k,\\
\sum_{k=-\infty}^{\infty}z_{2k+1}T^k = x_{2n+1} T^n +\sum_{k=n+1}^{\infty}\beta_k T^k.
\end{gather*}
It follows that $z_{2n}=x_{2n}$, $z_{2n+1}=x_{2n+1}$, and $z_{k}=0$ for all $k<2n$.

Note that $m\in\{2n,2n+1\}$, so $z_m=x_m\ne 0$. If $m$ is even, then clearly $m=\min\supp{z}$. If $m$ is odd, then $z_{m-1}=z_{2n}=x_{2n}=x_{m-1}=0$, hence $m=\min\supp{z}$ as well. Therefore, $\lvert( f^{-1}\circ L \circ f)(x)\rvert=\lvert z\rvert=c^m=\lvert x\rvert$.

\medskip

Let $g_f:=f^{-1}\circ g \circ f$. Since $f$ is additive, we have:
\[g_f(x)=f^{-1}(L(f(x))+\tau)=(f^{-1}\circ L \circ f)(x)+f^{-1}(\tau) \quad \text{for all }x\in \KK.\]
It now follows from \eqref{eq:f^-1Lf} that
$\lvert g_f(x)-g_f(y)\rvert = \lvert (f^{-1}\circ L \circ f)(x-y)\rvert=\lvert x-y\rvert$ for all $x,y\in \KK$, so $g_f$ is an isometry of $\KK$.
\end{proof}

Using \eqref{eq:estimation_f^-1}, we get the following property of the function $f$.

\begin{lem}\label{lem:SA-invariance}
Let $f\colon \KK\to \KK^2$ be as in \eqref{eq:def_of_f}. Then, for any $x_0\in \KK$ and $r>0$, there exists $\varepsilon\in (0,1)$ such that the each of the sets $f(B[x_0,r])$ and $f(S[x_0,r])$ is $\SA(2,D_{\KK},\varepsilon)$-invariant.
\end{lem}

\begin{proof}
Let $x_0\in \KK$ and $r>0$. Choose any $0<\varepsilon <\min\left\{\frac{\sqrt{r}}{\lVert f(x_0)\rVert_{\infty}+1},1\right\}$. It is sufficient to show that, for all $g\in \SA(2,D_{\KK},\varepsilon)$, the function $g_f:=f^{-1}\circ g\circ f$ maps each of the sets $B[x_0,r]$ and $S[x_0,r]$ onto itself.

Fix an arbitrary $g=(L,\tau)\in \SA(2,D_{\KK},\varepsilon)$.
It follows from \eqref{eq:def_of_SL(2,D_K,e)} that the linear mapping $L-I$, where $I$ is the identity on $\KK^2$, is $\varepsilon$-Lipschitz.
Hence, by \eqref{eq:estimation_f^-1}, we have the following estimation:
\begin{equation*}
  \begin{split}
    \lvert g_f(x_0) - x_0\rvert &= \lvert f^{-1}[g(f(x_0))-f(x_0)]\rvert \le \lVert g(f(x_0))-f(x_0)\rVert_{\infty}^{2}\\
  &=\lVert(L-I)(f(x_0))+\tau\rVert_{\infty}^{2}\le (\varepsilon\lVert f(x_0)\rVert_{\infty}+\varepsilon)^2<r.
  \end{split} 
\end{equation*} 

The function $g_f$ is, by Lemma \ref{lem:f^-1gf_is_isometry}, an isometry of $\KK$; so the images of $B[x_0,r]$ and $S[x_0,r]$ under $g_f$ are equal to $B[g_f(x_0),r]$ and $S[g_f(x_0),r]$, respectively. Since $\lvert g_f(x_0) - x_0\rvert<r$, we obtain $B[g_f(x_0),r]=B[x_0,r]$ (by the strong triangle property), as well as $S[g_f(x_0),r]=S[x_0,r]$ (by the isosceles property).
\end{proof}

\section{Paradoxical decompositions of complete discretely valued fields and their subsets}

We will now prove the main theorem about paradoxical decompositions of complete discretely valued fields and their specific subsets, such as balls and spheres, with respect to isometries.
Our approach is motivated by the analogous results \cite[Theorem 3.10]{KO-BBMS} for normed spaces $(\KK^n,\lVert\blank \rVert_{\infty})$ of dimension $n>1$.

\begin{thm}\label{th:main}
Let $\KK$ be a complete discretely valued field. Then $\KK$, all balls, and all spheres in $\KK$ are $\Iso{\KK}$-paradoxical using four pieces.
\end{thm}

\begin{proof}
It follows from Corollary \ref{cor:important} that the group action $\Iso{\KK}\acts \KK$ is isomorphic to $\Iso{\FF((T))} \acts \FF((T))$, where $\FF$ is the residue field of $\KK$ and the field $\FF((T))$ is equipped with the valuation \eqref{eq:abs_val_Laurent} for some $c\in (0,1)$. Therefore, we can assume without loss of generality that $\KK:=\FF((T))$.

Denote $G_{\varepsilon}:=\SA(2,D_{\KK}, \varepsilon)$ for $\varepsilon \in (0,1)$, and let $f\colon \KK \to \KK^2$ be defined by \eqref{eq:def_of_f}. The function $f^{-1}$ induces an isomorphism of the group actions $G_{\varepsilon}\acts E$ and $f^{-1}G_{\varepsilon}f\acts f^{-1}(E)$, for any $G_{\varepsilon}$-invariant subset $E\subseteq \KK^2$. By \cite[Theorem 3.9]{KO-BBMS}, every such $E$ is $G_{\varepsilon}$-paradoxical using four pieces. Hence, by Remark \ref{rem:isomorphic_actions}, $f^{-1}(E)$ is $f^{-1}G_{\varepsilon}f$-paradoxical using four pieces.

We know from Lemma \ref{lem:f^-1gf_is_isometry} that $f^{-1}G_{\varepsilon}f$ is a subgroup of $\Iso{\KK}$ for every $\varepsilon\in (0,1)$. Taking $E:=\KK^2$, we obtain the desired $\Iso{\KK}$-paradoxical decomposition of the whole $\KK$.

For any $x_0\in \KK$ and $r>0$, Lemma \ref{lem:SA-invariance} allows us to choose an $\varepsilon\in (0,1)$ such that $f(B[x_0,r])$ and $f(S[x_0,r])$ are $G_{\varepsilon}$-invariant. Substituting $E:=f(B[x_0,r])$ and $E:=f(S[x_0,r])$, we obtain the desired paradoxical decompositions of $B[x_0,r]$ and $S[x_0,r]$, respectively.
\end{proof}

\begin{rem}\label{rem:locally_commutative}
Keep the assumptions and notation of Theorem \ref{th:main}.
If $Z$ denotes $\KK$, any ball, or any sphere in $\KK$, there exists a free group of rank two $F_2\le \Iso{\KK}$ such that the action $F_2 \acts Z$ is locally commutative. Indeed, it suffices to take $\varepsilon\in (0,1)$ such that $f(Z)$ is $G_\varepsilon$-invariant, and use the fact \cite[Proposition~3.7]{KO-BBMS} that $G_\varepsilon$ contains a free subgroup of rank two $F'_2$ acting locally commutatively on $\KK^2$. Then $F_2:=f^{-1}F'_{2} f$ is as desired.
\end{rem}

We intend to prove that if $\KK$ is additionally separable, the sets from Theorem \ref{th:main} admit $\Iso{\KK}$-paradoxical decompositions using pieces having the Baire property. 

\begin{thm}\label{th:Baire_paradox}
Let $\KK$ be a separable complete discretely valued field. Then $\KK$, all balls, and all spheres in $\KK$ are Baire $\Iso{\KK}$-paradoxical using six pieces.
\end{thm}

\begin{proof}
  Let $Z$ denote $\KK$, any ball, or any sphere in $\KK$. Note that $Z$ is both closed and open in $\KK$; in particular, $Z$ is a Polish space.
  As in Remark \ref{rem:locally_commutative}, we can find a free group of rank two $F'_{2}\le G_{\varepsilon}$ acting locally commutatively on $f(Z)$. Then, the action of $F_2:=f^{-1}F'_{2} f\le \Iso{\KK}$ on $Z$ is locally commutative. 
  
  In order to apply \cite[Theorem 5.9]{DF}, we need to show that there exists an $F_2$-invariant set $Y\subseteq Z$ which is comeager in $Z$ and such that the restricted action $F_2 \acts Y$ is free.
  Consider any $g\in F'_{2} \setminus \{1\}$. Since $g$ is a nontrivial affine transformation of $\KK^2$, the set $\Fix{g}:=\{x\in f(Z): g.x=x\}$ is a subset of a line, so it is nowhere dense in $\KK^2$. Note that $f(Z)$ is open in $\KK^2$ because $f\colon \KK \to \KK^2$ is a homeomorphism. Hence, $\Fix{g}$ is nowhere dense also in $f(Z)$. Since $F'_{2}$ is countable, it follows that $M:=\bigcup_{g\in F'_{2} \setminus \{1\}} \Fix{g}$ is meager in $f(Z)$. Notice that $F'_{2}$ acts freely on $f(Z) \setminus M$, so $F_2$ acts freely on $Y:=Z\setminus f^{-1}(M)$.
  Since $f$ is a homeomorphism, the set $Y$ is comeager in $Z$.

  By \cite[Theorem 5.9]{DF}, $Z$ is Baire $F_2$-equidecomposable using six pieces that have the Baire property in $Z$. Since $Z\in \mathcal{B}(\KK)$, the pieces belong to $\mathcal{B}(\KK)$ as well.
\end{proof}

When $\KK$ is locally compact, we obtain the following result about equidecomposability, which can be called the ``strong'' Banach--Tarski paradox. For the proof of its Baire version, we follow the approach from \cite{GrMaPikhurko}, which relies on powerful results by Marks and Unger \cite{MU} regarding matchings in Borel graphs.

\begin{thm}\label{thm:strong_paradox}
  Let $\KK$ be a locally compact non-Archimedean valued field, and $A,B\subseteq \KK$ be bounded subsets with nonempty interiors. Then $A$ and $B$ are $\Iso{\KK}$-equidecomposable.
  Moreover, if $A,B\in \mathcal{B}(\KK)$, then $A$ and $B$ are Baire $\Iso{\KK}$-equidecomposable.
\end{thm}

\begin{proof}
  Let $E$ be a ball in $\KK$ containing both $A$ and $B$. By compactness of $E$, we have
  $E\subseteq g_1.A \cup \dots \cup g_n.A$ for some $n\in\mathbb{N}$ and translations $g_1,\dots, g_n$ (hence isometries) of $\KK$.
  The same works for $B$ in the place of $A$.
  Since $E$ is $\Iso{\KK}$-paradoxical (by Theorem \ref{th:main}), an application of \cite[Corollary 10.22]{W} yields that $A\sim_{\Iso{\KK}} E\sim_{\Iso{\KK}} B$; therefore, $A\sim_{\Iso{\KK}} B$.

  Suppose that additionally $A,B\in\mathcal{B}({\KK})$. Note that $\KK$ is a $\sigma$-compact complete metric space, so it is a Polish space. From the first part of the theorem we infer that the action $\Iso{\KK}\acts \KK$ satisfies the requirements of \cite[Definition~1.7]{GrMaPikhurko}. The Baire $\Iso{\KK}$-equidecomposability of $A$ and $B$ follows immediately from \cite[Proposition 1.8\,(ii)]{GrMaPikhurko}. 
\end{proof}

Finally, we apply Theorems \ref{th:main}, \ref{th:Baire_paradox}, and \ref{thm:strong_paradox} to the field $\QQ_p$ of $p$-adic numbers, for any prime $p$.

\begin{cor}\label{cor:p-adic}
Consider the field $(\QQ_p,\lvert \blank\rvert_p)$ of $p$-adic numbers for a prime $p$. Then $\QQ_p$, all balls, and all spheres in $\QQ_p$ are $\Iso{\QQ_p}$-paradoxical using four pieces and Baire $\Iso{\QQ_p}$-paradoxical using six pieces. Moreover, any two subsets $A,B\subseteq \QQ_p$ with nonempty interiors are $\Iso{\QQ_p}$-equidecomposable. If additionally $A,B\in \mathcal{B}(\QQ_p)$, then $A$ and $B$ are Baire $\Iso{\QQ_p}$-equidecomposable.
\end{cor}

\bibliographystyle{siam}
\bibliography{references}

@article{BT,
  author    = {Banach, Stefan and Tarski, Alfred},
  title     = {Sur la d{\'{e}}composition des ensembles de points en parties respectivement congruentes},
  journal   = {Fund. Math.},
  volume    = {6},
  number    = {1},
  pages     = {244--277},
  year      = {1924},
  publisher = {Polska Akademia Nauk. Instytut Matematyczny PAN},
  doi       = {10.4064/fm-6-1-244-277}
}

@book{W,
  author    = {Tomkowicz, Grzegorz and Wagon, Stan},
  title     = {The {B}anach--{T}arski paradox},
  edition   = {Second},
  publisher = {Cambridge University Press},
  address   = {New York},
  year      = {2016},
  isbn      = {978-1-107-04259-9}
}

@article{Tomkowicz2017,
  author  = {Tomkowicz, Grzegorz},
  title   = {Banach-{T}arski paradox in some complete manifolds},
  journal = {Proc. Amer. Math. Soc.},
  volume  = {145},
  year    = {2017},
  number  = {12},
  pages   = {5359--5362},
  doi     = {10.1090/proc/13657}
}

@article{DF,
  author  = {Dougherty, Randall and Foreman, Matthew},
  title   = {Banach-{T}arski paradox using pieces with the property of {B}aire},
  journal = {Proc. Nat. Acad. Sci. U.S.A.},
  volume  = {89},
  year    = {1992},
  number  = {22},
  pages   = {10726--10728},
  doi     = {10.1073/pnas.89.22.10726}
}

@article{GrMaPikhurko,
  author  = {Grabowski, {\L}ukasz and M\'ath\'e, Andr\'as and Pikhurko, Oleg},
  title   = {Measurable equidecompositions for group actions with an expansion property},
  journal = {J. Eur. Math. Soc. (JEMS)},
  volume  = {24},
  year    = {2022},
  number  = {12},
  pages   = {4277--4326},
  doi     = {10.4171/jems/1189}
}

@article{MU,
  author  = {Marks, Andrew and Unger, Spencer},
  title   = {Baire measurable paradoxical decompositions via matchings},
  journal = {Adv. Math.},
  volume  = {289},
  year    = {2016},
  pages   = {397--410},
  doi     = {10.1016/j.aim.2015.11.034}
}

@article{Weh,
  author  = {Wehrung, Friedrich},
  title   = {Baire paradoxical decompositions need at least six pieces},
  journal = {Proc. Amer. Math. Soc.},
  volume  = {121},
  year    = {1994},
  number  = {2},
  pages   = {643--644},
  doi     = {10.2307/2160449}
}

@article{KO-RACSAM,
  author  = {Orzechowski, K.},
  title   = {The {B}anach--{T}arski paradox for some subsets of finite-dimensional normed spaces over non-{A}rchimedean valued fields},
  journal = {Rev. Real Acad. Cienc. Exactas Fis. Nat. Ser. A-Mat.},
  volume  = {119},
  number  = {107},
  year    = {2025}
}

@article{KO-BBMS,
  author  = {Orzechowski, Kamil},
  title   = {Paradoxical decompositions of finite-dimensional non-{A}rchimedean normed spaces},
  journal = {Bull. Belg. Math. Soc. Simon Stevin},
  volume  = {32},
  year    = {2025},
  number  = {3},
  pages   = {377--390}
}

@article{FPS-survey,
  author  = {Sambale, Benjamin},
  title   = {{A}n {I}nvitation to {F}ormal {P}ower series},
  journal = {Jahresber. Dtsch. Math.-Ver.},
  volume  = {125},
  year    = {2023},
  number  = {1},
  pages   = {3--69},
  doi     = {10.1365/s13291-022-00256-6}
}

@book{UC,
  author    = {Schikhof, W. H.},
  title     = {Ultrametric calculus. An introduction to {$p$-adic} analysis},
  publisher = {Cambridge University Press},
  address   = {Cambridge},
  year      = {1984},
  pages     = {viii+306},
  doi       = {10.1017/CBO9780511623844}
}

@book{PG,
  author    = {Perez-Garcia, C. and Schikhof, W. H.},
  title     = {Locally convex spaces over non-{A}rchimedean valued fields},
  publisher = {Cambridge University Press},
  address   = {Cambridge},
  year      = {2010},
  doi       = {10.1017/CBO9780511729959}
}

@techreport{vRNotes,
  author      = {{van Rooij}, A. C. M.},
  title       = {Notes on {\(p\)-adic} {B}anach spaces},
  institution = {Department of Mathematics, University of Nijmegen},
  year        = {1976},
  type        = {Report},
  number      = {7633}
}

@book{Serre-LF,
  author    = {Serre, Jean-Pierre},
  title     = {Local fields},
  series    = {Graduate Texts in Mathematics},
  volume    = {67},
  note      = {Translated from French},
  publisher = {Springer-Verlag, New York-Berlin},
  year      = {1979},
  pages     = {viii+241}
}

@book{Schneider,
  author    = {Schneider, Peter},
  title     = {Nonarchimedean functional analysis},
  series    = {Springer Monographs in Mathematics},
  publisher = {Springer-Verlag},
  address   = {Berlin},
  year      = {2002},
  doi       = {10.1007/978-3-662-04728-6}
}

@article{Kubzdela,
  author  = {Kubzdela, Albert},
  title   = {Isometries, {M}azur-{U}lam theorem and {A}leksandrov problem for non-{A}rchimedean normed spaces},
  journal = {Nonlinear Anal.},
  volume  = {75},
  year    = {2012},
  number  = {4},
  pages   = {2060--2068},
  doi     = {10.1016/j.na.2011.10.006}
}

@article{Bishop,
  author  = {Bishop, Edward},
  title   = {Isometries of the {$p$}-adic numbers},
  journal = {J. Ramanujan Math. Soc.},
  volume  = {8},
  year    = {1993},
  number  = {1-2},
  pages   = {1--5}
}

@book{Kechris,
  author    = {Kechris, Alexander S.},
  title     = {Classical descriptive set theory},
  series    = {Graduate Texts in Mathematics},
  volume    = {156},
  publisher = {Springer-Verlag, New York},
  year      = {1995},
  pages     = {xviii+402},
  doi       = {10.1007/978-1-4612-4190-4}
}

@book{Kuratowski,
  author    = {Kuratowski, Casimir},
  title     = {Topologie. {V}ol. {I}},
  series    = {Monografie Matematyczne [Mathematical Monographs]},
  volume    = {Tom 20},
  publisher = {Pa\'{n}stwowe Wydawnictwo Naukowe, Warsaw},
  year      = {1958},
  pages     = {xiii+494},
  note      = {In French}
}

@article{TomkowiczBoundedSets,
  author   = {Tomkowicz, Grzegorz},
  title    = {On bounded paradoxical sets and {L}ie groups},
  journal  = {Geom. Dedicata},
  fjournal = {Geometriae Dedicata},
  volume   = {218},
  year     = {2024},
  number   = {3},
  note     = {Paper No. 72},
  issn     = {0046-5755,1572-9168},
  mrclass  = {03E05 (20M05 22E15 28C10)},
  mrnumber = {4732359},
  doi      = {10.1007/s10711-024-00923-1},
  url      = {https://doi.org/10.1007/s10711-024-00923-1}
}

\end{document}